\documentclass[12pt,draft]{amsart}
\usepackage[all]{xy}
\usepackage{amsfonts}
\usepackage{amssymb}

\usepackage{enumerate}

\usepackage{color}

\renewcommand{\le}{\leqslant}
\renewcommand{\ge}{\geqslant}

\newtheorem{teo}{Theorem}[section]
\newtheorem{lem}[teo]{Lemma}
\newtheorem{prop}[teo]{Proposition}

\newtheorem{dfn}[teo]{Definition}
\newtheorem{rk}[teo]{Remark}
\newtheorem{ex}[teo]{Example}

\def\<{\langle}
\def\>{\rangle}
\def\ss{\subset}
\def\sse{\subseteq}

\def\e{\varepsilon}

\def\C{{\mathbb C}}

\def\Z{{\mathbb Z}}

\def\Id{\operatorname{Id}}

\def\diag{\mathop{\rm diag}\nolimits}

\def\T{{\mathbb T}}
\def\1{\mathbf 1}

\newcommand{\Mat}[4]{\left( \begin{array}{cc}
                            #1 & #2 \\
                            #3 & #4
                      \end{array} \right)}

\def\cB{{\mathcal B}}
\def\cP{{\mathcal P}}

\def\B{{\mathbb B}}

\begin{document}

\title[Kuiper theorem]
{On Kuiper type theorems for uniform Roe algebras}

\author{Vladimir Manuilov}
\address{Dept. of Mech. and Math., Moscow State University,
119991 GSP-1  Moscow, Russia}
\email{manuilov@mech.math.msu.su}
\urladdr{
http://mech.math.msu.su/\~{}manuilov}

\author{Evgenij Troitsky}
\address{Dept. of Mech. and Math., Moscow State University,
119991 GSP-1  Moscow, Russia}
\email{troitsky@mech.math.msu.su}
\urladdr{
http://mech.math.msu.su/\~{}troitsky}

\begin{abstract}
Generalizing the case of an infinite discrete metric space of finite diameter, we say that
a discrete metric space $(X,d)$
is a Kuiper space, if the group of invertible elements of its uniform Roe algebra is norm-contractible.
Various sufficient conditions on $(X,d)$ to be or not to be
a Kuiper space are obtained.
\end{abstract}

\maketitle

\section*{Introduction}

The uniform Roe algebra $C^*_u(X)$ \cite{RoeBook,RoeBook2}, is a $C^*$-algebra associated with any discrete
metric space $X$ that encodes analytically the coarse geometry of the space. The uniform Roe algebras have attracted 
a lot of attention of specialists in Noncommutative Geometry, Elliptic Theory, and Harmonic Analysis on Groups in the last couple of decades. We refer to 
\cite{Brown-Ozawa,NowakYu2012Book,AraEtAl2018,LiLiao2018} and the bibliography therein.

The famous Kuiper Theorem \cite{Kui} states that the group of invertible (or unitary) operators on a Hilbert space is contractible. It has important applications in $K$-theory and in index theory \cite{AK}.

It is well-known that the algebra of all bounded operators on a separable Hilbert space can be viewed as the uniform Roe algebra of  an infinite  discrete metric space of finite diameter, so the Kuiper Theorem states that the group of invertibles in the uniform Roe algebra of an infinite discrete space of finite diameter is contractible. This leads to the following question: which discrete metric spaces share this property?
We call such spaces \emph{Kuiper spaces}. To the positive, we show  that a modification of the approach coming from the original Kuiper's proof can be used to show that if $X$ 
can be covered by balls of bounded radius having infinitely many points in each of them, 
then it is a Kuiper space. On the negative, we show that discrete metric spaces either being locally finite with the 
F\"olner property, satisfying an appropriate form
of the Morita equivalence, or 
having a decomposition generalizing the decomposition of the metric space of integers into the positive and the negative part are not Kuiper spaces.
Corollaries, examples, and related results are considered.
 Regretfully, we are still very far from having a criterion for being a Kuiper space. Note that the property of being a Kuiper space is not coarsely invariant: a single point space is not a Kuiper space, but is coarsely equivalent to an infinite space of finite diameter.

The Kuiper type theorems for invertibles in new algebras
were obtained recently
in \cite{ManFinAlg,TroManAlg}, see also \cite{ManGraf} for a related research.

The paper is organized as follows. In Section
\ref{sec:prelimmetr} we give necessary definitions and prove some facts about metric spaces and operators with finite propagation. In Section \ref{sec:firsstate} we give first positive and negative statements about Kuiper spaces.
In Section \ref{sec:kuiper_main} we prove the main positive and
in Section \ref{sec:main_negative} the main negative results.

\section{Preliminaries}\label{sec:prelimmetr}
Let $(X,d)$ be a (countable) discrete metric space.
Then the unit functions supported at one point $\delta_x$,
$x\in X$, form \emph{the standard base} of 
the corresponding $\ell^2$ space $\ell^2(X)$.  For 
 a bounded operator $F:\ell^2(X)\to \ell^2(X)$,  let $(F_{xy})_{x,y\in X}$ denote the matrix of $F$ with respect to the base $\{\delta_x\}_{x\in X}$.

\begin{dfn}\label{dfn:propagat}
\rm
Denote by $\cP(F)$ \emph{the propagation} of $F$, i.e.
 $\cP(F)=\sup\{d(x,z):x,z\in X,F_{xz}\ne 0\}$.
\end{dfn}

Note that the triangle 
inequality $d(x,y)\le d(x,z)+d(z,y)$ implies
\begin{equation}\label{eq:propa_ineq}
\cP(FG)\le \cP(F) + \cP(G).
\end{equation}

\begin{dfn}\label{dfn:uniRoe}
\rm
The $C^*$-algebra $C^*_u(X)$
generated by operators of finite 
propagation in the algebra $\B(\ell_2(X))$ of all bounded
operators is called \emph{the uniform Roe algebra}.  
\end{dfn}

\begin{rk}\label{rk:contrac_unitary_eq_invert}
\rm
Let us note that the contractibility
of the unitary group $U(C^*_u(X))$ is equivalent 
to the contractibility of
the group of invertibles of $C^*_u(X)$.
Indeed, since $C^*_u(X)$ is a $C^*$-algebra,
the formula $((1-t)\cdot\Id + t\cdot (F F^*)^{-1/2})F$
defines a deformation retraction of invertibles onto
unitaries. 
\end{rk}

\begin{dfn}
\rm
If $U(C^*_u(X))$ (equivalently, the group of invertibles of $C^*_u(X)$)
is contractible, we say that $(X,d)$ is a \emph{Kuiper space}.
\end{dfn}

\begin{dfn}\label{dfn:fpibs}
{\rm
We say that a discrete metric space $(X,d)$ is PIUBS
(has a countable partition by infinite uniformly bounded sets)
if there exists a sequence of its
points $\{x(k)\}_{k\in\mathbb N}$, a finite number $r>0$, 
and a collection of sets $D_k\subset X$ such that
\begin{enumerate}[1)]
\item $\{D_k\}_{k\in\mathbb N}$ is a partition of $X$;
\item $x(k)\in D_k \subseteq B_r(x(k))$ for each $k$ 
(in particular, $\{x(k)\}_{k\in\mathbb N}$ is a countable $r$-net for $X$), where $B_r(y)$ denotes the closed ball of radius $r$ centered at $y$;
\item each $D_k$ contains infinitely many points.
\end{enumerate}
}
\end{dfn}

Evidently PIUBS implies the following property. 

\begin{dfn}\label{dfn:fcib}
{\rm
We say that a discrete metric space $(X,d)$ is CIUBB
(has a cover by infinite uniformly bounded balls)
if  there exists a sequence of its
points $\{x(k)\}_{k\in\mathbb N}$ and a finite number $r>0$ such that
\begin{enumerate}[1.]
\item The balls $B_r(x(k))$, $k\in\mathbb N$, form a cover of $X$ (i.e.
$\{x(k)\}$ is an $r$-net for $X$).
\item Each ball $B_r(x(k))$, $k\in\mathbb N$, contains infinitely 
many points.
\end{enumerate}
}
\end{dfn}

Moreover, we have the following statement.

\begin{prop}\label{prop:fpibsfcib}
PIUBS is equivalent to CIUBB.
\end{prop}

\begin{proof}
We need to verify that CIUBB implies PIUBS. 
Define a subsequence $x(k(j))$ and sets $D_j\sse B_{3r}(x(k(j)))$
inductively in the following way.
We write $m\in C(k)$ if
$B_r(x(m))\cap B_r(x(k))=\varnothing$.
Take 
$$
k(1):=1,\qquad D_1:=X\setminus \cup_{m\in C(1)} B_r(x(m)).
$$
Then $B_r(x(1))\sse D_1 \sse B_{3r}(x(k(1)))$ and $D_1$ has
infinitely many points.
Now let $k(2)$ be the minimal element in the set $C(1)$.
Set
$$
D_2:=\cup_{m\in C(1)} B_r(x(m))\setminus \cup_{m\in C(k(2))} B_r(x(m)).
$$
Then $B_r(x(k(2)))\sse D_2 \sse B_{3r}(x(k(2)))$ 
and $D_2$ has
infinitely many points.
And so on.
The collection $D_i$, the sequence $x(k(i))$ and the radius $3r$
form the data, which shows that $X$ satisfies Definition \ref{dfn:fpibs}. 
\end{proof}

We will need the following definition.

\begin{dfn}\label{dfn:d_sparse_set}
\rm
We say that a subset $Y$ of $(X,d)$ is $r$-\emph{sparse},
if, for any $y\in Y$, $B_r(y)=\{y\}$.
\end{dfn}

Evidently, we have
\begin{lem}\label{lem:small_prop_on_d_sparse}
Suppose, $Y$ is an $r$-sparse subset of $(X,d)$ and 
$F:\ell^2(X)\to \ell^2(X)$ is an operator of propagation
$p<r$. Then, with the respect to the decomposition
$\ell^2(X)=\ell^2(X\setminus Y)\oplus \ell^2(Y)$,
the operator $F$ has the following matrix form:
$\Mat{F_1}00{D}$, where $D$ has a diagonal matrix with respect to the standard base.
\end{lem}

\section{First statements}\label{sec:firsstate}

We start with an evident negative statement.

\begin{prop}\label{prop:no_Kuiper_finite}
Suppose $(X,d)$ is a finite metric space.
Then the group of invertibles in $C^*_u(X)$
is not contractible.
\end{prop}

\begin{proof}
In this case $C^*_u(X)$ is the matrix
algebra $M_{n}(\C)$, where $n=|X|$, and its ivertibles
form the group $GL_{n}(\C)$, which is
homotopy equivalent to the unitary group
$U_{n}(\C)$. Its fundamental group is not
trivial (in fact $\cong \Z$) due to the
epimorphism $\det: U_{n}(\C)\to S^1\subset \C$.
\end{proof}

\begin{prop}\label{prop:Kuiper_bounded}
Suppose, $(X,d)$ is an infinite metric space
of finite diameter.
Then the group of invertibles in $C^*_u(X)$
is contractible.
\end{prop}

\begin{proof}
In this case $C^*_u(X)=\B(\ell_2(X))$ and
the statement is exactly the original
Kuiper theorem \cite{Kui}.
\end{proof}

Immediately from the definition of the uniform Roe
algebra we obtain the following
\begin{prop}\label{prop:coarsebij}
Suppose, $f:(X,d)\to (Y,\rho)$ is a bijection  that is
a coarse equivalence of metrics 
(i.e. there exist functions $\phi_1$ and $\phi_2$ on $[0,\infty)$ with $\lim_{t\to\infty}\phi_i(t)=\infty$, $i=1,2$, such that
$\phi_1(d(x,y))\le \rho(f(x),f(y)) \le \phi_2(d(x,y))$ for any $x,y\in X$).  
Then $C^*_u(X)\cong C^*_u(Y)$, in particular, $Y$ is a 
Kuiper space if and only if so is $X$.
\end{prop}

\begin{teo}\label{teo:kuip_inf_sparse}
Suppose, for any $r$, there exists a subspace $X_r$ of $(X,d)$
such that 
\begin{itemize}
\item[1)] $X_r$ is a Kuiper space;
\item[2)] $X\setminus X_r$ is $r$-sparse.
\end{itemize}
Then $X$ is a Kuiper space.
\end{teo}

\begin{proof}
The general argument based on the Atiyah
theorem ``on small balls'' reduces a proof of contractibility
of the group of invertibles  in  a Banach algebra 
$\cB\sse \B(H)$ to
the proof of the following fact: for any finite polyhedron $P$
with  vertices  $A_1,\dots,A_N$
and its inclusion $J: P\ss GL(\cB)$ there is a homotopy of $J$
to the constant mapping $P \to \1\in GL(\cB)$.
Moreover, for $\cB=C^*_u(X)$, we may take these vertices being of finite propagation. So,  assume that any point of $J(P)$
has propagation $\le p$, for some $p$.
By Lemma \ref{lem:small_prop_on_d_sparse}
and the second condition,
for $r>p$,
we have a map $J_1:P \to C^*_u(X_r)$ such that
$J(p)=\Mat{J_1(p)}00{\Id}$ with the respect to the
decomposition $\ell^2(X)=\ell^2(X_r)\oplus \ell^2(X\setminus X_r)$. It remains to use the first condition.
\end{proof}

\section{Main Kuiper-type theorem}\label{sec:kuiper_main}
In the next theorem
we will follow partially the schema of the Kuiper proof \cite{Kui}.
In \cite{ManFinAlg} and \cite{TroManAlg} it was adopted
for invertibles in special $C^*$-algebras.

\begin{teo}\label{teo:Kuiper_for_PIUBS}
If $X$ is PIUBS (or, equivalently CIUBB) then 
the group of invertibles in $C^*_u(X)$
is contractible.
\end{teo}

\begin{proof}
We reduce the problem to the study of $J(P)$,
as at the beginning of the proof of Theorem \ref{teo:kuip_inf_sparse}.

Also, as a result of an arbitrary small perturbation, we may
assume that the columns of the matrices of $A\in J(P)$ are of finite
length:
\begin{equation}\label{eq:step0}
\mbox{ for any }i,\mbox{ there exists }j(i)\mbox{ such that }
a_j^i=0 \mbox{ for }j>j(i).
\end{equation}
For $\cB=C^*_u(X)$,
the corresponding small linear homotopy lies in 
invertibles in
$\cB$, because
zero matrix elements remain zero and propagation does not increase.
So, any $F\in J(P)$ is an invertible operator of finite propagation $p$.

Choose $\e>0$ such that $\e$-neighborhood of $J(P)$
consists of invertibles.
Using (\ref{eq:step0}) and tending to $0$ of elements
of any row of the matrix of a vertex of $J(P)$,
one can choose inductively 
two non-intersecting sequences of distinct base vectors of 
$\ell^2(X)$: $\delta_{z(i)}$ and $\delta_{y(i)}$ such that 
\begin{itemize}
\item[a)] for each $i$, $z(i)$ and $y(i)$ are in the same $D_k\sse B_r(x(k))$;
\item[b)] all $\delta_{z(i)}$ and $F(\delta_{y(j)})$ are almost orthogonal for all $F\in J(P)$, $i\ne j$, i.e.
\begin{center}
$\<\delta_{z(i)},F(\delta_{y(j)})\>=0$, if $i>j$,
and $\<\delta_{z(i)},F(\delta_{y(j)})\> <\e/2^{i+j}$
 if $i<j$; 
\end{center}
\item[c)] moreover, if $i<j$, and the vector $F(\delta_{y(i)})$
has non-zero coordinates with numbers $v(i,1),\dots,v(i,N_i)$,
then $\<\delta_{v(i,s)},F(\delta_{y(j)})\> <\e/(N_i 2^{i+j})$,
$s=1,\dots N_i$;
\item[d)] $\{y(i)\}$ visits each $D_k\sse B_r(x(k))$ infinitely
many times.
\end{itemize}
After a small linear homotopy from $F$ to $F'$ (vanishing of some matrix elements)
and a passage to a subsequence
we can assume instead of b) and c) the following:
\begin{enumerate}
\item[b')] all elements $\delta_{z(i)}$ and $F'(\delta_{y(i)})$ are orthogonal, $F'(\delta_{y(i)})$ has finitely many non-zero
entries, and these entries are at distinct places for distinct
$y(i)$.  
\end{enumerate}
The homotopy is small in the following sense. 
We replace by zeroes some matrix elements of 
$F(\delta_{y(j)})$, namely
the coordinates with numbers $z(i)$, $v(i,1),\dots,v(i,N_i)$,
where $i<j$ (these sets of indexes may partially coincide).
The norm of distance between $F(\delta_{y(j)})$ and $F'(\delta_{y(j)})$ is less than the square root of
$$
\sum_{i=1}^{j-1}\left( \left(\frac{\e}{2^{i+j}}\right)^2  +
\sum_{s=1}^{N_i} \left(\frac{\e}{N_i 2^{i+j}}\right)^2\right)
\le
\frac{\e^2}{4^j} \frac{1}{2} \left(1 + \frac{1}{N_i}\right) \le
\frac{\e^2}{4^j}.
$$  
Hence, taking into account all $F(\delta_{y(j)})$, we obtain
that the distance between $F$ and $F'$ is less than $\e$.
In particular, the homotopy consists of invertibles. 
The homotopy is linear, hence continuous in $F\in J(P)$.

We obtain a new inclusion $J'$ of $P$. Conserve
the notation: $F\in J'(P)$.
Then we rotate first $F(\delta_{y(i)})$ to $\delta_{x(i)}$
and then $\delta_{x(i)}$ to $\delta_{y(i)}$.
Both (generalized) rotations have a block-diagonal form
with blocks having finite propagation. 
Indeed, if $F(\delta_{y(i)})$ has non-zero entries 
at places corresponding to $\tilde{y}_1,\dots,\tilde{y}_s$,
then $d(\tilde{y}_j,y(i))\le p$, $j=1,\dots,s$.
The corresponding block for the first rotation is
in the rows and columns corresponding to 
$x(i),\tilde{y}_1,\dots,\tilde{y}_s$
with the following explicit
form:
$$
\left(
\begin{array}{c|ccc}
\cos (t) &
 {\sin(t)}{b^{-1} \bar a_1} & \dots & {\sin(t)}{b^{-1} \bar a_s}\\
\hline
-\sin(t) a_1  &
 & & \\
\vdots & \multicolumn{3}{c}{\cos(t)\cdot pr_a}  \\
-\sin(t) a_s  &
  &  & \\
\end{array}
\right)
+\left(
\begin{array}{c|ccc}
0 &
 0 & \dots & 0\\
\hline
0  &
 & & \\
\vdots & \multicolumn{3}{c}{1- pr_a}  \\
0 &
  &  & \\
\end{array}
\right),
$$
where $(a_1,\dots, a_s)$ are non-zero entries of
$F(\delta_{y(i)})$, $b:=|a_1|^2+\cdots +|a_s|^2$, and
$$
pr_a=
\left(
\begin{array}{ccc}
a_1 b^{-1} \bar a_1 & \dots & a_1 b^{-1} \bar a_s \\
\vdots & \vdots & \vdots \\
a_s b^{-1} \bar a_1 & \dots & a_s b^{-1} \bar a_s 
\end{array}
\right)
$$
is the projection on the subspace generated by
$(a_1,\dots,a_s)$ (see Equation (8) in \cite{TroManAlg}).
Hence, 
$$
d(x(i),\tilde{y}_j)\le d(x(i),y(i))+d(y(i),\tilde{y}_j)
\le 2r + p
$$
and
$$
d(\tilde{y}_k,\tilde{y}_j)\le d(\tilde{y}_k,y(i))+d(y(i),\tilde{y}_j)
\le 2p.
$$
For the second rotation, its blocks evidently have 
propagation $\le 2r$.
Taking the
composition of these rotations with $F$ we obtain
a homotopy lying in the invertibles with propagation 
$\le (2r + 2p) +  2r + p=4r +3p$,
and with the final operator having the form
$\Mat 1{G_*}0{G}$ with respect to
the decomposition $H=H'\oplus (H')^\bot$,
where $H'$ is generated by $\delta_{y(i)}$.
Moreover, these rotation homotopies are continuous
in operator argument. Indeed, the distance
between blocks for $F$ and $F'$ at time $t$ is
estimated via the distance between corresponding
$F'(\delta_{y(i)})$ and $F(\delta_{y(i)})$, while
the distance between entire operators   
is the supremum over $i$ of these distances. 
The homotopy $\Mat 1{G_*\cdot t}0{G}$ of invertibles
does not increase propagation and is continuous in $G_*$. 

Now decompose $H'$ into infinitely many
orthogonal summands in the following way: 
enumerate the points related
to the $\delta_x$-base of $(H')^\bot$:
$u(1),u(2),\dots$ (i.e. $\{y(i)\}\cup \{u(j)\}=X$).
For each $u(j)$ choose $x(k(j))$ such that 
$u(j)\in D_{k(j)}\sse B_r(x(k(j))$ and denote by $w(j,m)$, $m=1,2,\dots$, infinitely many distinct
$y(i)$ from the same $D_{k(j)}\sse B_r(x(k(j))$. 
Thus
\begin{equation}\label{eq:ma_est}
d(u(j),w(j,m))<2r.
\end{equation}
The remaining points
(i.e. other than $u(j)$ and $w(j,m)$, $j,m=1,2,\dots$) denote
by $v(i)$, $i=1,2,\dots$. Suppose, $H''$ is generated by
$\{\delta_{v(i)}\}$, $H_0=(H')^\bot$ is generated by $\{\delta_{u(j)}\}$, 
$H_m$, $m\ge 1$, is generated by $\{\delta_{w(j,m)}\}$
($m$ is fixed). Then we have the following
orthogonal decomposition:
$$
\ell^2(X)=H''\oplus H_0 \oplus H_1 \oplus H_2 \oplus
\cdots.
$$
Our operator has the following diagonal form
with respect to this decomposition:
$$
\diag(1,G,1,1,\dots).
$$
Denote by $G_m$ as well, an operator in $H_m$
with the same matrix with respect to $\{w(j,m)\}$, as
$G$ has with respect to $\{u(j)\}$.
To unify the notation we will denote $G$ by $G_0$
and $u(j)$ by $w(j,0)$ in general
formulas some later. 
Similarly for $G^{-1}$. 
Then we define in a well-known way the following
two homotopies.
First, connect 
$\Mat 1001 = \Mat {G_m G_m^{-1}}001 $
at each summand  $H_m \oplus H_{m+1}$,
$m=1,3,5,7,\dots$, with
$ \Mat {G_m^{-1}} 00 {G_{m+1}}$ via the homotopy
$$
\Mat {\cos t} {-\sin t} {\sin t} {\cos t}
\Mat {G_m} 001
\Mat {\cos t} {\sin t} {-\sin t} {\cos t}
\Mat {G_m^{-1}} 001,\qquad t\in [0,\pi/2].
$$
More accurately we should write this down as
\begin{multline*}
\Mat {\cos t} {-\sin t \cdot J_{m,m+1}} 
{\sin t\cdot J_{m+1,m}} {\cos t}
\Mat {G_m} 001\\
\Mat {\cos t} {\sin t \cdot J_{m,m+1}} 
{-\sin t \cdot J_{m+1,m}} {\cos t}
\Mat {G_m^{-1}} 001,
\end{multline*}
where $J_{m',m''}:H_{m''}\to H_{m'}$,
$J_{m',m''}: w(j,m'')\mapsto w(j,m')$,
and hence, 
\begin{equation}\label{eq_m_m+1}
J_{m,m'} G_m J_{m',m}=G_{m'},\quad
\mbox{ in particular, }\quad
J_{m,0} G J_{0,m}=G_{m}.
\end{equation}
Denote the rotation matrices:
\begin{equation}\label{eq:rota}
R_{m',m''}(t):=\Mat {\cos t} 
{\sin t \cdot J_{m',m''}} 
{-\sin t \cdot J_{m'',m'}} {\cos t}.
\end{equation}
If we extend this homotopy to be constant $=G$ on
$H_0$ and $=1$ on $H''$, we obtain a path connecting
$$
\diag (1,G,1,1,1,\dots)\quad\mbox{ and }\quad
\diag (1,G, G_1^{-1}, G_2, G_3^{-1}, G_4 \dots)
$$
with respect to the above decomposition. 
Now we connect
the restriction $\Mat {G_m} 00{G_{m+1}^{-1}}$
onto $H_m\oplus H_{m+1}$
(where $m=0,2,4,\dots$)
with the 
identity with the help of the following homotopy:
\begin{multline*}
\Mat {\cos t} {-\sin t\cdot J_{m,m+1}} 
{\sin t\cdot J_{m+1,m}} {\cos t}
\Mat {G_m^{-1}} 001\\
\Mat {\cos t} {\sin t\cdot J_{m,m+1}} 
{-\sin t\cdot J_{m+1,m}} {\cos t}
\Mat {G_m} 001, 
\end{multline*}
$t\in [0,\pi/2]$.
Taking the direct sum we obtain the
homotopy
$$
\diag (1,G, G_1^{-1}, G_2, G_3^{-1}, G_4 \dots)
\sim \diag (1,1, 1, 1, 1, 1,\dots).
$$ 

Let us prove that the above homotopies
lie in the uniform Roe algebra $C^*_u(X)$. 
For this purpose it is sufficient to verify that
\begin{enumerate}[1)]
\item $\cP(\diag(1,G,G_1,1,G_3,1,\dots))<\infty$
and $\cP(\diag(1,G,1,G_2,1,G_4,1,\dots))<\infty$,
\item $\diag(1,G,(G_1)^{-1},1,(G_3)^{-1},1,\dots)\in C^*_u(X)$ and \newline
$\diag(1,(G_0)^{-1},1,(G_2)^{-1},1,\dots)\in C^*_u(X)$,
\item 
$\cP(\diag(1,1,R_{1,2}(t),R_{3,4}(t),\dots))<
\infty$
and $\cP(\diag(1,R_{0,1}(t),R_{2,3}(t),\dots)
< \infty$ (this works for the left side rotations
as well after $t\leftrightarrow -t$).
\end{enumerate}
The inverse operator $G^{-1}$ may not have 
finite propagation, but only belongs to
the uniform Roe algebra. Denote by $G^{(n)}$
its approximation with finite propagation $p_n$
(the sum of $n$ first summands of the corresponding
series). If 
\begin{equation}\label{eq_m_m+1_inv}
G^{(n)}_m:=J_{m,0} G^{(n)} J_{0,m},
\end{equation}
then, by (\ref{eq_m_m+1}) it approximates
$(G_m)^{-1}$ uniformly in $m$ (this approximation
does not need propagation properties). Thus, the
corresponding block diagonal operators approximate
operators in item 2) above.
Hence, by (\ref{eq_m_m+1}), (\ref{eq_m_m+1_inv})
and (\ref{eq:rota}), to prove 1)-3) above, it is
sufficient to prove that $\cP(J_{m,m'})$
is uniformly bounded in $m$ and $m'$. 

The matrix of $J_{m,m'}$ has the following form:
$$
(J_{m,m'})_{xy}=\left \{ 
\begin{array}{ll}
1,&\mbox{ if }  y=w(j,m'') \mbox{ and } x=w(j,m')
\mbox{ for some }j,  \\
0,&\mbox{ otherwise}.
\end{array}
\right.
$$
 As 
$$
d(y,x)=d(w(j,m''),w(j,m'))\le
d(w(j,m''),u(j))+d(w(j,m'),u(j)) < 2r + 2r =4r
$$
(see (\ref{eq:ma_est})),  we have  $\cP(J_{m,m'})\le 4r$,
and we are done.

Finally, we observe that all families of
operators are continuous in $G$, because $G\to G^{-1}$
is continuous, since $G$ runs over a compact set in invertibles,
and the other factors in formulas do not depend on $G$.
\end{proof}

\begin{rk}
{\rm
If we consider a finite partition in the above definition,
we arrive immediately to the case of Prop. \ref{prop:Kuiper_bounded}.
}
\end{rk}

\section{When Kuiper Theorem fails}\label{sec:main_negative}
In this section we consider several cases when the unitary group of the uniform Roe algebra $C^*_u(X)$ fails to be contractible.

For a metric space $X$ let $\sqcup^nX$ denote the space $X_1\sqcup\ldots\sqcup X_n$, where $\alpha_i:X_i\to X$, $i=1,\ldots,n$, are isometries, with the metric given by $d(x,y)=d_X(\alpha_i(x),\alpha_j(y))+|i-j|$, where $x\in X_i$, $y\in X_j$.

\begin{lem}
$C^*_u(\sqcup^nX)\cong M_n(C^*_u(X))$.
\end{lem}

\begin{proof}
Obvious.
\end{proof}

We call $X$ {\em stable} if for any $n\in\mathbb N$ there exists a bijection $\beta_n:\sqcup^n X\to X$ which is a coarse equivalence of metrics. For stable $X$, $\beta_n$ induces an isomorphism $M_n(C^*_u(X))\cong C^*_u(X)$ for any $n\in\mathbb N$
(see Prop. \ref{prop:coarsebij}).

$X$ is {\em locally finite} (or {\em proper}) if each ball contains a finite number of points. For a subset $Y\subset X$ set 
$$
\partial_RY=\{x\in X:d(x,Y)<R;d(x,X\setminus Y)<R\}. 
$$
Recall that $X$  satisfies the {\em F\"olner property}  if for any $R>0$ and any $\varepsilon>0$ there exists a finite subset $F\subset X$ such that $\frac{|\partial_RF|}{|F|}<\varepsilon$. 

If $X$ is locally finite then, for $T\in C^*_u(X)$ and for a finite set $F\subset X$ put $f_F(T)=\frac{1}{|F|}\sum_{x\in F}T_{xx}$. Given a sequence of finite sets $F_n\subset X$ and an ultrafilter $\omega$ on $\mathbb N$, one can define the ultralimit $\lim_{\omega}f_{F_n}(T)$. The F\"olner property allows to define in this way a trace $f$ on $C^*_u(X)$ with $f(1)=1$ (for details see \cite{Brown-Ozawa}, \cite{NowakYu2012Book}). 

\begin{prop}\label{P2}
Let $X$ be  a  stable, locally finite metric space with the 
F\"olner property. Then $X$ is not a Kuiper space.
\end{prop}   
\begin{proof}
If the group $GL(C^*_u(X))$ of invertible elements is contractible then, by stability of $X$, so are $GL(M_n(C^*_u(X)))$ for any $n\in\mathbb N$. By \cite{Wood}, cf. also \cite{KarCli}, 
$$
K_0(A)=\pi_1(\injlim_{n\to\infty}GL(M_n(A)))
$$ 
for any unital Banach algebra $A$, hence  $K_0(C^*_u(X))=0$. But $f(1)\neq f(0)$, hence $[1]\neq[0]$ in $K_0(C^*_u(X))$.   
\end{proof}

The following proposition is almost folklore.

\begin{prop}\label{P1}
Let $X=A\sqcup B$, $|A|=|B|=\infty$. Suppose that
\begin{itemize}
\item[(a)]
for any $R>0$ the set $\{x\in A:d(x,B)<R\}$ is finite;
\item[(b)]
there exists $C>0$, a bijection $\alpha:X\to X$, such that $d(x,\alpha(x))<C$ for any $x\in X$, and $x_0\in A$ such that $\alpha|_A$ is a bijection between $A$ and $A\setminus\{x_0\}$.
\end{itemize}
Then $U(C^*_u(X))$ is not path connected, hence  $X$ is not a Kuiper space. 
\end{prop}
\begin{proof}
By assumption, $l_2(X)=l_2(A)\oplus l_2(B)$, and we can write operators as two-by-two matrices with respect to this decomposition, $T=\left(\begin{matrix}T_{11}&T_{12}\\T_{21}&T_{22}\end{matrix}\right)$. If $T$ has finite propagation then, by (a), $T_{12}$ and $T_{21}$ are compact, hence the same is true for any $T\in C^*_u(X)$.

Let $V\in \mathbb B(l_2(X))$ be the unitary determined by the bijection $\alpha$. By (b), $V\in U(C^*_u(X))$. If $U(C^*_u(X))$ would be path connected
then there would exist a continuous path $V(t)$, $t\in[0,1]$, in $U(C^*_u(X))$ such that $V(1)=V$, $V(0)=1$. By (a), $T_{11}(t)$ is Fredholm for any $t\in[0,1]$. Hence $V_{11}(t)$ is a homotopy of Fredholm operators. But, by (b), $\operatorname{Index}V_{11}(1)=1\neq 0=\operatorname{Index}V_{11}(0)$. The contradiction proves the claim. 
\end{proof}

The following example shows that the F\"olner condition is a too strong requirement for  not being a Kuiper space. 

\begin{ex}
\rm
Let $X_0=\{0\}$, $X_1=\{1,2\}$, $X_2=\{3,4,5,6\}$, $X_3=\{7,8,9,10,11,12,13,14\}$, \ldots, and let $-n\in X_{-k}$ if $n\in X_k$. We write $k(n)=k$ if $n\in X_k$.

For $n,m\in\mathbb Z$, $n\neq m$, set $d(n,m)=|k(n)-k(m)|+1$. This gives a metric on $X=\mathbb Z=\sqcup_{k\in\mathbb Z}X_k$, which makes $X$ a locally bounded uniformly discrete space of unbounded geometry. Note that the map $\beta:n\mapsto [n/2]$ satisfies 
\begin{itemize}
\item[(1)]
$\beta^{-1}(n)$ consists of two points for any $n\in X$;
\item[(2)]
$d(\beta(n),n)\leq 2$ for any $n\in X$.
\end{itemize}  
 Therefore, $(X,d)$ has a paradoxical decomposition property: we can write $X=U\sqcup V$, where $U\cap\beta^{-1}(n)$ is a single point for any $n\in X$. Then there are bijections $\beta_U:X\to U$ and $\beta_V:X\to V$ such that $d(\beta_U(n),n)\leq 2$ and $d(\beta_V(n),n)\leq 2$ for any $n\in X$. Therefore, $X$ doesn't satisfy the F\"olner condition. This can be seen directly: for $R=2$ let $F\subset X$ be a finite subset, and let $s,t$ be such that $F\subset X_s\sqcup\ldots\sqcup X_t$, and $F\cap X_s$ and $F\cap X_t$ are not empty. Let $r=\max(|s|,|t|)$. Then 
$$
|F|\leq |X_s|+\cdots+|X_t|\leq 2(2^{r+1}-1)\leq 2^{r+2}, 
$$
and $\partial_R F$ contains either $X_r$ or $X_{-r}$, hence $|\partial_R F|\geq 2^r$, and we have $\frac{|\partial_R F|}{|F|}\geq \frac{1}{4}$. Thus, $X$ doesn't satisfy the assumptions of Proposition \ref{P2}.

Another decomposition $X=A\sqcup B$, where $A=\{0,1,2,\ldots\}$, $B=\{-1,-2,\ldots\}$ shows that $X$ satisfies the assumptions of Proposition \ref{P1}, where $\alpha$ is the two-sided shift on $X$.
\end{ex}

Our second example shows that the assumptions of Proposition \ref{P1} are also too strong:

\begin{ex}
\rm
Let $X=\mathbb Z^n$. It is stable and satisfies the F\"olner condition, hence satisfies the assumptions of Proposition \ref{P2}, but for $n\geq 2$ it cannot be decomposed into two subspaces $A\sqcup B$ such that (a) from Proposition \ref{P1} holds.
\end{ex}

 
\end{document}